\documentclass[12pt]{amsart}
 \usepackage{graphicx}
\usepackage{color}
\usepackage{float}
\usepackage{subfig}
\usepackage{amssymb}
 \usepackage{amsmath}
 \usepackage{amscd}
 \usepackage{amsfonts}
 \usepackage{mathrsfs}
 \usepackage{amsthm}
 \usepackage{hyperref}
\usepackage{fullpage}
\usepackage[small,nohug]{diagrams}
\usepackage{mathtools}

\theoremstyle{plain}
\newtheorem{thm}{Theorem}[section]
 \newtheorem{lem}[thm]{Lemma}
 \newtheorem*{lem*}{Lemma}
 \newtheorem{prop}[thm]{Proposition}
 \newtheorem{cor}[thm]{Corollary}
 
\theoremstyle{definition}
 \newtheorem{definition}[thm]{Definition}
 \theoremstyle{remark}

\diagramstyle[labelstyle=\scriptstyle]

\newcommand{\al}{\alpha}

\newcommand{\inv}{^{-1}}

\newcommand{\sg}{\sigma}

\DeclarePairedDelimiter\ceil{\lceil}{\rceil}
\DeclarePairedDelimiter\floor{\lfloor}{\rfloor}

\newcommand{\Id}{\textrm{Id}}

\newcommand{\spn}{\textrm{span}}

\newcommand{\lb}{\langle}
\newcommand{\rb}{\rangle}
\newcommand{\wt}{\widetilde}

\newcommand{\bmat}[1]{\left[\begin{array}{#1}}
\newcommand{\emat}{\end{array}\right]}

\title{A bound on the codimensions of a PI-algebra using group geometry}
    
\author{Christopher S. Henry}

\begin{document}
\begin{abstract} In this note we  draw a connection between noncommutative algebra and geometric group theory. Specifically, we ask whether it is possible to bound the sequence of codimensions for an associative PI-algebra using techniques from geometric group theory. The classic and best known bound on codimension growth was derived by finding a ``nice'' spanning set for the multilinear polynomials of degree $n$ inside the free algebra. This spanning set  corresponds  to permutations in the symmetric group $S_n$ which are so-called  \emph{$d$-good}, where $d$ is the degree of the identity satisfied by the algebra. The motivation for our question comes from the fact that there is an obvious relationship between the word metric on $S_n$ and the property of being $d$-good. We answer in the affirmative, by finding a spanning set that corresponds to permutations which are large with respect to the word metric. We provide an explicit algorithm and formula for calculating the size  of the resulting bound, and demonstrate that it is asymptotically worse than the classic one.

\noindent{\bf AMS MSC classes:} 16R99,20F10.

\end{abstract}

\maketitle

\section{Introduction}

Let $A$ be an associative polynomial identity algebra (PI-algebra) over a field $k$, satisfying an identity of degree $d$. It is a well known result of Regev (\cite{R-1978}) that the sequence of codimensions  $\{c_n(A)\}$  is bounded exponentially in $n$. The bound comes from counting the number of $d$-good permutations in the symmetric group $S_n$, for which Dilworth has shown that there are at most $(d-1)^{2n}$. One can identify permutations  with  monomials in the subspace $P_n \subset k\lb X\rb$ of multilinear polynomials of degree $n$, where $k\lb X\rb =k\lb x_1, x_2, \dotsc \rb$ is the free algebra on countably many variables. The existence of an identity of degree $d$ implies that  $P_n$ is spanned - modulo the identities of $A$ - by $d$-good monomials, which gives the result.

In general, the strategy for bounding the codimensions can be described as follows: identify a collection of  monomials  which form a spanning set for $P_n$ modulo the identities of $A$, and then calculate the size of this spanning set as $n$ grows large. The strategy itself is important because it has been adapted in several ways to provide arguments which establish the \emph{existence} of an identity for certain algebras; see as examples \cite{BGR-98}, \cite{BZ-98}, \cite{GZ-2005}, and \cite{PR-13}. 
Roughly speaking, in such arguments applied to associative algebras one still counts the number of $d$-good monomials in some  appropriate free object.  In the Lie algebra case however, the notion of $d$-indecomposable is used in place of $d$-good, and a  more complicated function is required to count the number of $d$-indecomposable monomials. It is therefore natural to ask if there are alternative ways of finding such a spanning set and estimating its size.

In this note,  we ask whether it is possible to find a spanning set for $P_n$ that is based on the geometry of the symmetric group, i.e. viewing $S_n$ as a metric space in the spirit of geometric group theory. The motivation for this question stems from the fact that both the word metric  and ``$d$-goodness'' provide a measure of the extent to which a given permutation is out of order. In fact, we find an easily stated relationship between the two concepts in Lemma \ref{lem-dgoodsize}.
One might hope to exploit this relationship to provide an improved bound on the codimensions compared to the one given by $d$-good monomials.

We find mixed results. On one hand, we are able to develop a novel strategy for bounding codimensions that is purely based on  geometry. The spanning set we derive corresponds to  monomials which are ``large'', that is,  those for which the associated permutation is a certain distance from the identity. Furthermore, we present an explicit formula and a nice algorithm for calculating the size of our bound. On the other hand, we confirm that our bound is asymptotically worse than the one provided by Regev, as it grows on the order of $n!$ as opposed to exponentially. 




\section{Preliminaries}
We present the classic (i.e. Regev's) bound on the codimensions of a PI-algebra, and then introduce some notation required for our results. 


\subsection{Regev's bound on codimension}
We follow the exposition in Chapter 4.2 of \cite{GZ-2005}. Denote by $S_n$ the symmetric group on $n$ elements and let $\sg \in S_n$. For $d \in \{2, \dotsc, n\}$ we say that $\sg$ is $d$-bad if there exists indices $1 \leq i_1 < \dotsc < i_d \leq n$ such that $\sg(i_1) >  \dotsc > \sg(i_d)$. If $\sg$ is not $d$-bad then we say it is $d$-good. Dilworth originally  provided the bound on the number of $d$-good permutations in a well known result.

\begin{lem} Let $d \in \{2 , \dotsc , n\}$. The number of $d$-good permutations is bounded above by $(d-1)^{2n}$.\end{lem}

Let $A$ be an associative $k$-algebra with $k$ a field, and $k\lb X \rb =k \lb x_1, x_2, \dotsc \rb$ be the free $k$-algebra on countably many variables. The set $P_n=  \spn\{ x_{\sg(1)}\cdots x_{\sg(n)} \mid \sg \in S_n\}$ is the subspace  of multilinear polynomials of degree $n$. A monomial $x_{\sg}=x_{\sg(1)} \cdots x_{\sg(n)} \in P_n$ is called $d$-good if the associated permutation $\sg$ is $d$-good, and similarly for $d$-bad. 

The algebra $A$ is a polynomial identity algebra (PI-algebra) if there exists some $f=f(x_1,\dotsc,x_m) \in k\lb X\rb$, such that $f(a_1,\dotsc , a_m) =0$ for all  $a_i \in A$.  
Denote by $\Id(A)$ the $T$-ideal of identities of $A$. The $n$-th \emph{codimension} of $A$ is given by $c_n(A) = \dim \frac{P_n}{P_n \cap \Id(A)}$. Note that $A$ satisfies a polynomial identity of degree $n$ as soon as $c_n(A) < n!$.  

\begin{thm}\label{classicbound} Let $A$ be a PI-algebra satisfying an identity of degree $d$. Then for $n\geq d$ we have $c_n(A) \leq (d-1)^{2n}$, i.e. $P_n$ is spanned (modulo the identities of $A$) by $d$-good monomials.
\end{thm}
\begin{proof} We may assume that $A$ satisfies an identity of the form $$x_1 \cdots x_d = \sum_{1 \neq \tau \in S_d} \al_{\tau}x_{\tau(1)}\cdots x_{\tau(d)}. \quad \quad \quad(*)$$ Pick $x_{\sg}=x_{\sg(1)}\cdots x_{\sg(n)}$, minimal in the dictionary order on monomials, such that $x_{\sg}$ is not a linear combination of $d$-good monomials. In particular $\sg$ must be $d$-bad. By definition there are indices $1 \leq i_1 < \cdots < i_d \leq n$ with $\sg(i_1) > \cdots > \sg(i_d)$. We decompose $x_{\sg}$ by setting $w_0 = x_{\sg(1)} \cdots x_{\sg(i_1 -1)}$, $w_1 = x_{\sg(i_1)} \cdots x_{\sg(i_2 -1)}$, $\dotsc, w_d = x_{\sg(i_d)} \cdots x_{\sg(n)}$. Clearly $x_{\sg} = w_0 \cdots w_d$, and for any $1\neq \tau \in S_d$ we have $w_0w_{\tau(1)} \cdots w_{\tau(d)} < w_0w_1\cdots w_d= x_{\sg}$ in the dictionary order. By minimality this means that each $w_0w_{\tau(1)} \cdots w_{\tau(d)}$ is a linear combination of $d$-good monomials. Left multiply $(*)$ by the variable $x_0$, and use the resulting identity  to conclude that $x_{\sg}$ must also be a linear combination of $d$-good monomials, a contradiction.   
\end{proof}

As an application  we recall Regev's theorem regarding the tensor product of PI-algebras.  Regev's theorem demonstrates how codimension arguments are used to establish the existence of an identity for a particular algebra. It follows almost immediately from the following, which we state without proof. 

\begin{thm} Let $A$ and $B$ be PI-algebras. We have that $c_n(A \otimes B) \leq  c_n(A)c_n(B)$. 
\end{thm}

\begin{thm}[Regev] If $A$ and $B$ are PI-algebras, then so is $A \otimes B$. 
\end{thm}
\begin{proof} Assume that $A$ satisfies an identity of degree $d_1$, and  $B$ an identity of degree $d_2$. By the above and Theorem \ref{classicbound}, we have that $c_n(A\otimes B) \leq c_n(A)c_n(B) \leq (d_1 -1)^{2n} (d_2 - 1)^{2n}$. Let $k_1 = (d_1 -1)^2$ and $k_2 = (d_2 -1 )^2$, so that  $c_n(A\otimes B) \leq (k_1k_2)^n$. There exists an $m$ such that $(k_1k_2)^m  < m!$, and so $A \otimes B$ satisfies an identity of degree $m$.    
\end{proof}


\subsection{Geometry of $S_n$}\label{sec-geomSn}
A reference for this material is  \cite{E+}. We denote by $\bar{n}$ the set $\bar{n} = \{1,\dotsc, n\}$. For any $\sg \in S_n$, construct the descent set   $$R_{\sg} = \{ (i,j) \in \bar{n} \times \bar{n} \mid i < j \text{ and } \sg(j)< \sg(i) \}.$$ This is also known in the literature as the set of inversions. It should be clear that $R_{\sg}$ is uniquely determined by $\sg$. Furthermore, a subset $R \subset \bar{n} \times \bar{n}$ comes from a permutation precisely when:
\begin{enumerate}
	\item $(i,j)\text{ and } (j,k) \in R \implies (i,k) \in R$, and
	\item $(i,k) \in R \implies (i,j) \in R \text{ or } (j,k) \in R$ for all $i<j<k$.
\end{enumerate}

Take the generating set of $S_n$ to be $T= \{t_1, \dotsc, t_{n-1}\}$ where $t_i = (i,i+1)$. We may form the Cayley graph $\Gamma=\Gamma(S_n,T)$, which induces a metric on $S_n$ called the word metric. More specifically for $\sg , \tau \in S_n$, $d(\sg,\tau)$ is the length of the shortest path in $\Gamma$ from $\sg$ to $\tau$, i.e. the shortest expression of $\sg\inv\tau$ in terms of the generators. We take the ball of radius $K$ in $S_n$ to  be $B(K) = \{ \sg \in S_n \mid d(\sg,1) < K \}$; the complement of such a ball is $\widehat{B}(K) = \{\sg \in S_n \mid d(\sg,1)\geq K\}$. 

We use the symbol $\#$ for the cardinality of a set, so the number of elements in $B(K)$ is denoted $\# B(K)$, and similarly for $\# \widehat{B}(K)$. Finally, we denote the size of $\sg \in S_n$ by $|\sg| = d(\sg, 1)$, and remark that $|\sg| =  \# R_{\sg}$, the number of pairs in the descent set.  


\subsection{Some convenient notation}\label{sec-notation} As above, we use the shorthand $x_{\sg}$ for the monomial $x_{\sg(1)} \cdots x_{\sg(n)} \in P_n$. We wish to consider submonomials of $x_{\sg}$ which are often referred to as subwords. Choose another letter for such a subword, for example $w$, so that $w = x_{\sg(i)} \cdots x_{\sg(j)}$ for some $1\leq i \leq j \leq n$. The \emph{length} of this subword is $l(w)=j-i + 1$. We say that another subword $u$ precedes $w$, denoted $u\preceq w$, if $u = x_{\sg(i)} \cdots x_{\sg(j')}$ for $j' \leq j$.

For a decomposition of $x_{\sg}$ into subwords $x_{\sg}=w_1 \cdots w_k$ with $k\leq n$, we can act on $x_{\sg} $ by $\tau \in S_k$. We denote such an action by $x_{\tau(\sg)}= w_{\tau(1)} \cdots w_{\tau(k)}$, where $\sg':=\tau(\sg) \in S_n$ is the resulting permutation  of the indices. Finally, for $i \in \bar{n}$ we write $i \in w_j$  if $\sg(i)$ shows up as an index in $w_j$. 


\section{Results}

The motivation for our  results comes from the following simple observation.

\begin{lem}\label{lem-dgoodsize} Let $\sg \in S_n$. If $\sg \in B(\frac{d(d-1)}{2})$, then $\sg$ is $d$-good. 
\end{lem}
\begin{proof} We establish the contrapositive. If $\sg$ is $d$-bad, there are indices $i_1 < \cdots < i_d$ with $\sg(i_1) > \cdots > \sg(i_d)$. For each $j=1,\dotsc ,d-1$ there is a  pair $(i_j, i_l) \in R_{\sg}$ for every $l>j$, for a total of at least $\sum_{i=1}^{d-1} i = \frac{d(d-1)}{2}$  pairs in $R_{\sg}$. Hence, $\sg \in \widehat{B}(\frac{d(d-1)}{2})$.  
\end{proof}

Both size and ``$d$-badness'' provide  a  measure  of the extent to which a given permutation is out of order. 
Ideally one would like to state something like the following: if $A$ is a PI-algebra satisfying an identity of degree $d$, then $P_n$ is spanned by  $d$-good monomials $x_{\sg}$ with $\sg \in B(K_n)$, for some sequence of radii $K_n$. If we can make $K_n$ small enough, i.e. so that there are $d$-good permutations outside of $B(K_n)$, this would automatically provide an improved bound on the codimensions compared to the classic one.

The sticking point is that we require a decomposition of the monomial $x_{\sg}$ - as in the proof of Theorem \ref{classicbound} 
- that is somehow related to  $|\sg|$. There is a natural  such decomposition called \emph{left greedy form}, which we describe below in Sections \ref{sec-lgf} and \ref{sec-colgf}. However, there is no clear relationship between the $d$-good/bad condition and left greedy form. 
We are still able to obtain a bound on the codimensions, but it is not as small as one might hope.  


\subsection{Left greedy form}\label{sec-lgf}
\begin{definition} Let $\sg \in S_n$ with descent set $R_{\sg}$, and $x_{\sg} \in P_n$ the associated monomial. We define the \emph{initial chunk} $c$ of $x_{\sg}$ as follows. If $R_{\sg}= \emptyset$ (i.e. $\sg = 1$) then $c$ is the empty word.  Otherwise $c$ is the subword $x_{\sg}=w_0cw_1$ with $c=x_{\sg(i_0)} \cdots x_{\sg(j_0)}$, and $i_0, j_0 \in \bar{n}$ satisfying:
	\begin{enumerate}
		\item $(i_0,j)\in R_{\sg}$ for some $j\in \bar{n}$, and  $(i_0,j)$ is minimal in $R_{\sg}$ with respect to the dictionary order.
		\item $(i,j_0) \in R_{\sg}$ for some $i\in \bar{n}$,
		\item for all $i_0 \leq i\leq j_0$,  if $j > j_0$ then $(i,j) \notin R_{\sg}$.
	\end{enumerate}
\end{definition}

We remark that $w_0$ or $w_1$ (or both) may be empty words. For example the element $\delta \in S_n$ of maximal size $|\delta| =$$n\choose 2$ has initial chunk $x_{\delta}=x_{\delta(1)} \cdots x_{\delta(n)}=x_nx_{n-1}\cdots x_1 = c$. The following lemma describes the algorithm for finding the initial chunk. 

\begin{lem}\label{lem-constructlgf} For any $\sg \in S_n$ and $x_{\sg} \in P_n$, the initial chunk $c$ exists. 
\end{lem}
\begin{proof} If $\sg =1$ then we are done. Otherwise by assumption there is some $(i, j) \in R_{\sg}$ which is minimal in the dictionary order,  so take $i_0 =i$. We may set $j=j_0^1$ to be our candidate for  $j_0$.  For all  $i'$ with $i_0 \leq i' \leq j_0^1$, check whether there exists some $j' > j_0^1$ with $(i', j') \in R_{\sg}$. If no such $j'$ exists, then we have $j_0^1=j_0$. Otherwise, update $j_0^2=j'$, and repeat the process for $i''$ with $i' \leq i'' \leq j_0^2$. The sequence $j_0^1, j_0^2, \dotsc $ must terminate, since the monomial is of finite length.   
\end{proof}

Given $x_{\sg}=w_0cw_1$, we can proceed to find the initial chunk of $w_1$. Iterating this process gives a decomposition $x_{\sg}=w_0c_1 w_1c_2 \cdots w_{k-1}c_kw_k$, where $c_i$ is the initial chunk of $w_{i-1}\cdots c_kw_k$. Again, note that  $w_j$ may be empty for any $0\leq j\leq k$, however  we assume that $c_j$ is always nonempty (provided $\sg\neq 1$).  In the spirit of \cite{E+}, we refer to this decomposition as the \emph{left greedy form} of $x_{\sg}$. By construction, left greedy form satisfies two useful properties, which we now list.

\begin{prop}\label{prop-lgffacts} Let $x_{\sg}=w_0c_1 \cdots w_{k-1}c_kw_k \in P_n$ be in left greedy form.
	\begin{enumerate}
		\item If $i \in w_l$ for some $0 \leq l \leq k$, then $(i,j) \notin R_{\sg}$ for all $j \in \bar{n}$. 
		\item If $(i,j) \in R_{\sg}$, then $i, j \in c_{\sg}^l$ for some $1\leq l \leq k$. 
	\end{enumerate}
\end{prop}
\begin{proof}
	For 1, assume that $i \in w_l$  for some $l$ and that $(i,j) \in R_{\sg}$ for some $j \in \bar{n}$. This contradicts $c_{l+1}$ being the initial chunk of $w_{l}c_{l+1} \cdots c_kw_k$. Part 2 follows from part 1, and the definition of the initial chunk. 
\end{proof}

\begin{definition} Let $x_{\sg}=w_0c_1 \cdots c_kw_k$ be in left greedy form. 
	We say that a decomposition of $x_{\sg}$ \emph{preserves chunks} if it is of the form $$x_{\sg} = y^0_1 \cdots y^0_{m_0}c_1'\cdots  c_{k'}'y^{k'}_1 \cdots y^{k'}_{m_{k'}},$$ where 
\begin{enumerate} 
	\item $k'\leq k$,
	\item $c_i \preceq c_i'$ for all $i =1,\dotsc , k'$, and
	\item $y^{k'}_1 \cdots y_{m_{k'}}^{k'}$ is the empty word if $k' <k$.
\end{enumerate}
\end{definition}
Note that for such a decomposition and for $i<k'$, $l(c_i')$ is  is minimal when  $c_i' = c_i$, and maximal when  $c_i'=c_iw_i$. In other words, the decomposition distinguishes all the initial chunks, except possibly at the end of the word. 
We can act on a decomposition preserving chunks by elements of $S_{k' + m}$, where $m=\sum m_j$, as described in Section \ref{sec-notation}. More specifically, relabel the indices to be $\{1,\dotsc, k'+m\}$ and then apply $\tau$ to the indices.  The following corollary to Proposition \ref{prop-lgffacts} is key to the proof of our main theorem.

\begin{cor}\label{cor-lgfaction} Let $x_{\sg} = y^0_1 \cdots y^0_{m_0}c_1' \cdots c_{k'}'y^{k'}_1 \cdots y^{k'}_{m_{k'}} \in P_n$ be a decomposition preserving chunks, and assume that $x_{\sg} \neq c_1$. Take $m=\sum m_j$. For any $1\neq \tau \in S_{k'+m}$, we have $|\sg'| > |\sg|$ where $x_{\sg'} =x_{\tau(\sg)}$.  
\end{cor}
\begin{proof}
	First note that $|\sg'| \geq |\sg|$, since by part 2 of Proposition \ref{prop-lgffacts},  any pair $(i,j) \in R_{\sg}$ belongs to some chunk $c_l$. Under $\tau$ these indices get sent to  $i\mapsto i'$ and $j\mapsto j'$ respectively, and their order is preserved, i.e. $i' < j'$. But then we have $\sg'(j')=\sg(j) < \sg(i) = \sg'(i')$, and $(i',j') \in R_{\sg'}$. 
	
	 Therefore we must show that there is at least one additional pair belonging to $R_{\sg'}$. Since $\tau \neq 1$ there is some first (reading left to right) subword that is moved by $\tau$. We have two cases:\\
	\emph{Case 1:} Assume that $y^l_m$ is the first subword moved by $\tau$, and pick any $i \in y^l_m$. Since $y^l_m$ is moved, there is some subword $w$ in the given decomposition of $x_{\sg}$ that follows $y^l_m$, and gets moved to the position of $y^l_m$ under $\tau$. Pick $j \in w$. Since $w$ follows $y^l_m$ we have $i < j$. By part 1 of Proposition \ref{prop-lgffacts} $(i,j) \notin R_{\sg}$, i.e. $\sg(i) < \sg(j)$. Under $\tau$ the indices $i$ and $j$ get sent to $i\mapsto j'$ and $j \mapsto i'$ respectively, with $i'<j'$. But then  $\sg'(j') =\sg(i) < \sg(j)=\sg'(i')$, so $(i',j') \in R_{\sg'}$.\\
	\emph{Case 2:} Assume that $c_l'$ is the first subword moved by $\tau$, and pick any $i \in c_l'$. Since $c_l'$ is moved, there is some subword $w$ in the given decomposition of $x_{\sg}$ that follows $c_l'$, and is moved to the position of $c_l'$ under $\tau$.  Pick $j \in w$.  There are two subcases, depending on whether $i \in c_l$. If so, then by definition of left greedy form we must have that $(i,j) \notin R_{\sg}$, otherwise $j$ would belong to the chunk $c_l$. If not, then $i\in w_l$ and  again by part 1 of Proposition \ref{prop-lgffacts} we have $(i,j) \notin R_{\sg}$. In either case  $\sg(i) < \sg(j)$. As above, under $\tau$ the indices $i$ and $j$ get sent to $j'$ and $i'$, so that $\sg'(j')=\sg(i) < \sg(j)=\sg'(i')$ and $(i',j') \in R_{\sg'}$.  	
\end{proof}

\subsection{Constraints on left greedy form}\label{sec-colgf}

	To prove our main theorem, we need a few more facts regarding the left greedy form of monomials $x_{\sg}$ when $\sg$ is ``small'' in the geometry of $S_n$.
First we note that the number of chunks must be bounded. 

\begin{prop}\label{prop-numchunk} Let $x_{\sg} = w_0c_1 \cdots w_{k-1}c_kw_k \in P_n$ be in left greedy form, and $\sg \in B(K)$ for some $K\geq 1$.  Then $k < K$, i.e. the number of chunks is at most $K-1$.  
\end{prop}
\begin{proof}
	Since each chunk $c_l$ contains at least one pair $(i,j) \in R_{\sg}$. 
\end{proof}
We wish to use this fact to bound the length of chunks in left greedy form. First we need the following.  

\begin{lem}\label{lem-sizechunk} Let $x_{\sg} \in P_n$, $(i,j) \in R_{\sg}$, and $w$ be the subword $w=x_{\sg(i)} \cdots x_{\sg(j)}$. Then $|\sg| \geq l(w) - 1$. 
\end{lem}
\begin{proof}
	Since $(i,j) \in R_{\sg}$ we get (at least) one  pair in $R_{\sg}$ for each $i'$ with $i<i'<j$, by the properties of $R_{\sg}$ (see Section \ref{sec-geomSn}). That is, we get an additional $j-i -1$ pairs in $R_{\sg}$, for a total of at least $j-i-1 + 1 = j-i+1 - 1 = l(w) -1$. 
\end{proof}

This can be used to bound the length of chunks in the left greedy form of $x_{\sg}$ when $\sg$ is small. 


\begin{cor}\label{cor-sizechunk} Let $x_{\sg}=w_0c_1 \cdots w_{k-1}c_kw_k \in P_n$ be in left greedy form, and $\sg \in B(K)$ for some $K \geq 1$. Then $\sum l(c_j) < 2K$. 
\end{cor}
\begin{proof} Let $c_j = x_{\sg(i_0)} \cdots x_{\sg(j_0)}$, and denote by $|c_j|$ the number of pairs in $R_{\sg}$ belonging to $c_j$. Since all pairs in $R_{\sg}$ must show up in some $c_j$ then $\sum |c_j| = |\sg|$.  By the construction of left greedy form (Lemma \ref{lem-constructlgf}), we can find a sequence of pairs $$(i_0,j_1), (i_1,j_2), \dotsc , (i_m,j_{m+1} = j_0)$$ belonging to $R_{\sg}$, such that: 
	\begin{enumerate} 
		\item $i_0 < i_1 < \cdots < i_m$,
		\item $j_1 < j_2 < \cdots < j_{m+1}$,
		\item $i_l\leq j_l$ for all $1\leq l \leq m$.
	\end{enumerate}
	Using  Lemma \ref{lem-sizechunk}, we then have
	\begin{align*} |c_j| & \geq  (j_1 - i_0) + (j_2 - i_1) + \cdots (j_0 - i_m) \\
		& \geq (j_1 - i_0) + (j_2 - j_1) + \cdots + (j_0 - j_m)  \\
		& = j_0 - i_0 = j_0 -i_0 +1 - 1 \\
		& = l(c_j) - 1. 
	\end{align*}
But then $\sum l(c_j) \leq \sum |c_j| + k = |\sg| +k < K + K =  2K$. 
\end{proof}

\subsection{Main theorem}

Finally we come to our main result.

\begin{thm}\label{thm-main} Let $A$ be a PI-algebra satisfying an identity of degree $d$, and let $K_n=\frac{(n-d)}{2}$. Then for $n \geq d$  we have $c_n(A) \leq \# \widehat{B}(K_n)$.
\end{thm}
\begin{proof}
As in Theorem \ref{classicbound}, we may assume an identity of the form $$x_1 \cdots x_d = \sum_{1 \neq \tau \in S_d} \al_{\tau}x_{\tau(1)}\cdots x_{\tau(d)} \quad \quad \quad(*)$$
 We  show directly that any monomial $x_{\sg}\in P_n$ with $\sg \in B(K_n)$ can be written as a linear combination of monomials $x_{\sg'}$ with $\sg' \in \widehat{B}(K_n)$. It will suffice to show that any such $x_{\sg}$ can be written as a linear combination of $x_{\sg'}$  with $|\sg'| > |\sg|$.  

Let $x_{\sg} = w_0c_1 \cdots w_{k-1}c_kw_k$ be in left greedy form with $\sg \in B(K_n)$, and observe that $\sum l(c_j) + \sum l(w_j) =n$. By Corollary \ref{cor-sizechunk} we have that $\sum l(c_j) < 2K_n = n-d$. We would like to utilize Corollary \ref{cor-lgfaction}, for which we need that $k + \sum l(w_j) \geq d$. Indeed we have 
\begin{align*} k + \sum l(w_j) 
	& \geq \sum l(w_j) &&\\ 
	&= n - \sum l(c_j) &&\\
	&> n -  (n-d) &&  \\
	& = d.  && 
\end{align*}
This inequality says that we can always find a decomposition preserving chunks, of the form $x_{\sg} = y^0_1\cdots y^0_{m_0}c_1' \cdots c_{k'}'y^{k'}_1 \cdots y^ {k'}_{m_{k'}}$, where $k' + \sum m_j = d$. 

Using ($*$) and Corollary \ref{cor-lgfaction}, we then have that $x_{\sg} = \sum_{1\neq \tau  \in S_d} \al_{\tau} x_{\tau(\sg)} $, with $\tau(\sg) = \sg' $ satisfying $|\sg'| > |\sg|$.

\end{proof}

It is clear that one could improve the bound on the codimensions if there was some way to increase the radius of $\widehat{B}(K_n)$ in the above result. In particular the maximum size of an element in $S_n$ is $n \choose 2$$ = \frac{n(n-1)}{2}$, so if one could replace $K_n$ with a sequence that grows $\mathcal{O}(n^2)$ instead of just $\mathcal{O}(n)$, there may be some hope to provide a bound that is asymptotically better than $(d-1)^{2n}$.  However using the tools we have developed in this note, $K_n$ appears to be the best we can do.  


\section{Calculating the bound}
In this section we give an indication of how to compute $\# \widehat{B}(K_n)$, and provide a comparison to the classic bound of $(d-1)^{2n}$. Recall that $K_n = \frac{(n-d)}{2}$. 


\subsection{Algorithm and formula}
For a given $d$ and $n$, there is a nice algorithm for computing $\#\widehat{B}(K_n)$ that is easy to describe. It is  known - see for e.g. in \cite{DlH-2000}, result originally due to O. Rodrigues dating to 1838 - that one can count the number of permutations in $S_n$ with a given number of descents/inversions. Specifically, given the generating set $T=\{t_1, \dotsc, t_{n-1}\}$ and distance metric referenced in  Section \ref{sec-geomSn}, then $$\prod_{i=1}^{n-1} (1 + z + \cdots z^i) = \sum_{k=0}^{n \choose 2} I_n(k)z^k,$$ where the coefficient $I_n(k)$ counts the number of elements in $S_n$ of size $k$. Hence, one can simply multiply out the polynomial on the left hand side to determine $I_n(k)$, and then 
\begin{align}
	\#\widehat{B}(K_n) = \sum_{k=\ceil{K_n}}^{n\choose 2} I_n(k).
\end{align}

Furthermore, since $K_n < n$  one can use Knuth's formula for $I_n(k)$  to  provide an explicit formula for $\#\widehat{B}(K_n)$. Recall that the \emph{pentagonal numbers} are defined as $u_j = \frac{j(3j-1)}{2}$. 
\begin{prop}(Knuth) For $k \leq n$, 
\begin{align}\label{eq-Ink}I_n(k) = \binom{n+k-1}{k} + \sum_{j=1}^{\infty} (-1)^j \binom{n+k -u_j -j -1}{k-u_j -j} + \sum_{j=1}^{\infty} (-1)^j \binom{n+k -u_j -1}{k-u_j }.
\end{align}
\end{prop}
The sums on the right hand side converge, as the binomial coefficients are defined to be zero when the lower index is negative. Using Formula 2 above we then take 
\begin{align}\#\widehat{B}(K_n) = n! - \sum_{k=0}^{\floor{K_n}} I_n(k).
\end{align}

\subsection{Asymptotics}
In \cite{M-2001}, asymptotic estimates for $I_n(k)$ with $k\leq n$ were derived using Formula \ref{eq-Ink}. These provide a convenient functional form which we use to demonstrate that our bound is asymptotically worse than $(d-1)^{2n}$. 
\begin{thm}[\cite{M-2001}] For $k \leq n$ and $n$ large,
	\begin{align*} I_n(n-k) \approx \wt{I}_n(n-k)= \frac{2^{2n-k-1}}{\sqrt{n\pi}}Q(1 + \mathcal{O}(n\inv)),
	\end{align*}
	where $Q=\prod_{j=1}^{\infty} (1 - \frac{1}{2^j}) < 1$. 
\end{thm}
Setting $K=\floor{K_n}$, we get a lower bound $\phi(n)$ for $\#\widehat{B}(K_n)$, 
\begin{align*} \#\widehat{B}(K_n) & \approx n! - \sum_{k =K}^n \wt{I}_n(n-k)\\ 
	& > n! - \sum_{k=K}^{n} 2^{2n-k-1} \\
	& = n! - (2^{2n-K} -2^{n-1})\\
	&:= \phi(n)\\
	&>(d-1)^{2n},
\end{align*}
where the last inequality holds as  $n!\in \mathcal{O}(\phi(n))$.

\subsection{Small codimensions}
One possible advantage of the bound we have provided is that it is sharper for ``small'' codimensions. More precisely, let $n(d)=n_d$ be the least integer for which $(d-1)^{2n_d} < n_d!$. This is an increasing function of $d$. 

\begin{figure}[H]
    \centering
    \caption{Smallest $n$ such that $(d-1)^{2n} < n!$}
    \label{fig-nd}	
    \includegraphics[scale=0.55]{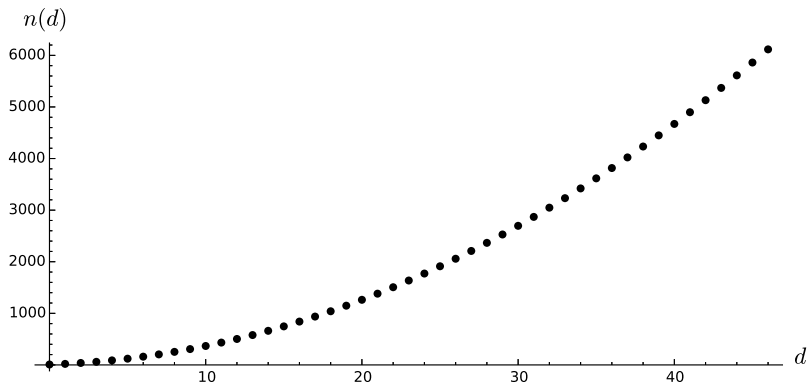}
    \vspace{-3mm}
\end{figure}

For $m$ such that $d \leq m < n(d)$, observe that the effective bound on the codimension $c_m(A)$ given by $d$-good monomials  is therefore actually  $m!$. However, we know simply due  to the constraints of geometry that $\#\widehat{B}(K_m)$ is strictly less than $m!$ for all $m>d$. 
This suggests that one should take the bound on $c_n(A)$ to be $\#\widehat{B}(K_n)$ until it surpasses $(d-1)^{2n}$. 

 \bibliographystyle{amsplain}

\end{document}